\documentclass[12pt]{amsart}

\usepackage{amsmath}
\usepackage{amsthm}
\usepackage{amssymb}
\usepackage{amscd}
\usepackage{mathrsfs}
\usepackage{txfonts}
\newtheorem{theorem}{Theorem}[section]
\newtheorem{definition}[theorem]{Definition}
\newtheorem{proposition}[theorem]{Proposition}
\newtheorem{lemma}[theorem]{Lemma}
\newtheorem{corollary}[theorem]{Corollary}

\newtheorem{remark}[theorem]{Remark}
\newtheorem{example}[theorem]{Example}

\newcommand{\C}{\mathbb C}

\newcommand{\Z}{\mathbb Z}

\newcommand{\SL}{\mathit{SL}}
\begin{document}

\title[A polynomial defined by Reidemeister torsion]
{A polynomial defined 
by the $\SL(2;\C)$-Reidemeister torsion 
for a homology 3-sphere obtained by Dehn-surgery along a torus knot}
\author{Teruaki Kitano}

\address{Department of Information Systems Science, 
Faculty of Science and Engineering, 
Soka University, 
Tangi-cho 1-236, 
Hachioji, Tokyo 192-8577, Japan}

\email{kitano@soka.ac.jp}



\begin{abstract}
Let $M_n$ be a homology 3-sphere 
obtained by $\frac1n$-Dehn surgery along a $(p,q)$-torus knot. 
We consider a polynomial $\sigma_{(p,q,n)}(t)$ 
whose zeros are the inverses of the Reideimeister torsion of $M_n$ for $\SL(2;\C)$-irreducible representations. % under some normalization. 
We give an explicit formula of this polynomial by using Tchebychev polynomials of the first kind. Further we also give a 3-term relations of these polynomials.
\end{abstract}

\maketitle
%%%%%%%%%%%%%%%%%%%%%%%%%%%%%
\section{Introduction}

Let $T(p,q)$ be a $(p, q)$-torus knot in $S^3$. 
Here $p,q$ are coprime and positive integers.  
%
%derived an explicit formula for the Reidemeister torsion of homology 3-spheres 
%obtained by $\frac1n$-Dehn surgeries along a $(p,q)$-torus knot $T(p,q)$ 
%for $\SL(2;\C)$-irreducible representations. 
%
Let $M_n$ be a homology 3-sphere obtained 
by $\frac1n$-Dehn surgery along $T(p,q)$. 
It is well known that 
$M_n$ is a Brieskorn homology 3-sphere $\Sigma(p,q,N)$ 
where we write $N$ for $|pqn+1|$. 
Here $\Sigma(p,q,N)$ is defined 
as 
\[
\{(z_1,z_2,z_3)\in \C^3\ |\ z_1^p+z_2^q+z_3^N=0,\ |z_1|^2+|z_2|^2+|z_3|^2=1\}. 
\]

In this paper we consider the Reidemeister torsion $\tau_\rho(M_n)$ of $M_n$ 
for an irreducible representation $\rho:\pi_1(M_n)\rightarrow\SL(2;\C)$. 
 
In the 1980's Johnson \cite{Johnson} 
gave an explicit formula for any non-trivial value of $\tau_{\rho}(M_n)$. 
Furthermore, % in the case of $(p,q)=(2,3)$,  
he proposed to consider 
the polynomial whose zero set coincides 
with the set of all non-trivial values $\{\frac{1}{\tau_\rho(M_n)}\}$, 
which is denoted by ${\sigma}_{(2,3,n)}(t)$.
Under some normalization of ${\sigma}_{(2,3,n)}(t)$,  
he gave a 3-term relation 
among ${\sigma}_{(2,3,n+1)}(t),{\sigma}_{(2,3,n)}(t)$ and ${\sigma}_{(2,3,n-1)}(t)$ 
by using Tchebychev polynomials of the first kind. 
%It is an analogy of the Dehn-surgery formula of the Casson invariant. 

Recently in \cite{Kitano15-2} 
we gave one generalization of the Johnson's formula 
for a $(2p', q)$-torus knot. 
Here $p', q$ are coprime odd integers. 
In this paper, we show the formula for any torus knot $T(p,q)$. 

\flushleft{Acknowledgements.}
The author was staying in Aix-Marseille University when he wrote this article. 
He thanks for their hospitality. 
This research was partially supported by JSPS KAKENHI 25400101. 
%%%%%%%%%%%%%%%%%
\section{Definition of Reidemeister torsion}

First let us describe definitions and properties 
of the Reidemeister torsion for $\SL(2;\C)$-representations. 
See Johnson \cite{Johnson}, Kitano \cite{Kitano94-1,Kitano94-2} 
and Porti \cite{Porti15-1} for details.

Let $\textbf{b}=(b_1,\cdots,b_d)$ and
$\textbf{c}= (c_1,\cdots,c_d)$ be two bases for a $d$-dimensional vector space $W$ 
over $\C$. 
Setting $\displaystyle b_i=\sum_{j=1}^d p_{ji}c_{j}$, 
we obtain a nonsingular matrix $P=(p_{ij})\in\mathit{GL}(d;\C)$. 
Let $[\textbf{b}/\textbf{ c}]$ denote the determinant of $P$.

Suppose
\[
C_*: 0\rightarrow C_k
\overset{\partial_k}{\rightarrow} 
C_{k-1}
\overset{\partial_{k-1}}{\rightarrow} 
\cdots 
\overset{\partial_{2}}{\rightarrow}
C_1\overset{\partial_1}{\rightarrow} C_0\rightarrow 0
\]
is an acyclic chain complex of finite dimensional vector spaces over $\C$. 
We assume that a preferred basis $\textbf{c}_i$ for $C_i$ is given for each $i$. 
That is, $C_\ast$ is a based acyclic chain complex over $\C$. 

Choose any basis $\textbf{b}_i$ for $B_i=\mathrm{Im}(\partial_{i+1})$ 
and take a lift of it in $C_{i+1}$, 
which is denoted by $\tilde{\textbf{b}}_i$.
Since $B_i=Z_i=\mathrm{Ker}{\partial_{i}}$, 
the basis $\textbf{b}_i$ can serve as a basis for $Z_i$. 
Furthermore since the sequence
\[
0\rightarrow Z_i
\rightarrow C_i
\overset{\partial_{i}}{\rightarrow} B_{i-1}\rightarrow 0
\]
is exact, the vectors $(\textbf{b}_i,\tilde{\textbf{b}}_{i-1}) $ form a basis for $C_i$. 
Here $\tilde{\textbf{b}}_{i-1}$ is a lift of $\textbf{b}_{i-1}$ in $C_i$. 
It is easily shown that
$[\textbf{b}_i,\tilde{\textbf{b}}_{i-1}/\textbf{c}_i]$ does not depend 
on a choice 
of a lift $\tilde{\textbf{ b}}_{i-1}$. 
Hence we can simply denote
it by $[\textbf{b}_i, \textbf{b}_{i-1}/\textbf{c}_i]$.

\begin{definition}
The torsion $\tau(C_*)$ 
of a based chain complex $C_*$ with $\{\textbf{c}_\ast\}$ is given by the alternating product 
\[
\tau(C_*)=\prod_{i=0}^k[\textbf{b}_i, \textbf{b}_{i-1} /\textbf{c}_i]^{(-1)^{i+1}}.
\]
\end{definition}

\begin{remark}
It is easy to see that $\tau(C_\ast)$ does not depend on choices of the bases 
$\{\textbf{b}_0,\cdots,\textbf{b}_k\}$. 
\end{remark}

Now we apply this torsion invariant of chain complexes to geometric situations 
as follows. 
Let $X$ be a finite CW-complex and $\tilde X$ a universal covering of $X$ 
with the lifted CW-complex structure. 
The fundamental group $\pi_1 X$ acts on $\tilde X$ from the right-hand side as deck transformations. 
We may assume that this action is free and cellular 
by taking a subdivision if we need. 
Then the chain complex $C_*(\tilde{X};\Z)$ has the structure of a chain complex of free $\Z[\pi_1 X]$-modules. 

Let $\rho:\pi_1 X\rightarrow \SL(2; \C)$  be a representation. 
We denote the 2-dimensional vector space $\C^2$ by $V$. 
Using the representation $\rho$, $V$ admits the structure of a $\Z[\pi_1 X]$-module 
and then we denote it by $V_\rho$.

Define the chain complex $C_*(X; V_\rho)$ 
by $C_*({\tilde X}; \Z)\otimes_{\Z[\pi_1 X]} V_\rho$ 
and choose a preferred
basis
\[
(\tilde{u}_1\otimes \textbf{e}_1, \tilde{u}_1\otimes \textbf{e}_2, \cdots,\tilde{u}_d\otimes\textbf{e}_1, \tilde{u}_d\otimes\textbf{e}_2)
\]
of $C_i(X; V_\rho)$ 
where 
$\{\textbf{e}_1 , 
\textbf{e}_2\}$ is a canonical basis of $V=\C^2$, 
$\{u_1,\cdots,u_d \}$ are the $i$-cells 
giving a basis of $C_i(X; \Z)$ and 
$\{\tilde{u}_1,\cdots,\tilde{u}_d\}$ are lifts of them on $\tilde{X}$. 

Now we suppose that $C_*(X; V_\rho)$ is acyclic, 
namely all homology groups $H_*(X; V_\rho)$ are vanishing. 
In this case $\rho$ is called an acyclic representation. 

\begin{definition}
Let $\rho:\pi_1(X)\rightarrow \SL(2; \C)$ be an acyclic representation. 
Then the Reidemeister torsion $\tau_\rho(X)\in\C\setminus\{0\}$ is defined by the torsion $\tau(C_*(X; V_\rho))$ of $C_*(X; V_\rho)$. 
\end{definition}

\begin{remark}
\noindent
\begin{enumerate}
\item
We define $\tau_\rho(X)=0$ for a non-acyclic representation $\rho$.
\item
The definition of $\tau_\rho(X)$ depends on several choices. 
However it is well known that it is a piecewise linear invariant 
in the case of $\SL(2;\C)$-representations. 
\end{enumerate}
\end{remark}

%%%%%%%%%%%%%%%%%%%%%%%%%%%%
\section{Johnson's theory}

%We apply the above proposition to a 3-manifold obtained by Dehn-surgery along a knot. 
Let $T(p,q)\subset S^3$ be a $(p, q)$-torus knot with coprime integers $p,q$. 
%Further let $N(K)$ be an open tubular neighborhood of $K$ 
%and $E(K)$ the knot exterior $S^3\setminus {N}(K)$. 
%We denote its closure of ${N(K)}$ by $\overline{N}$ which is homeomorphic to $S^{1}\times D^{2}$.
Now we write $M_n$ to a closed orientable 3-manifold 
obtained by a $\frac1n$-Dehn surgery along $T(p,q)$. 
%Naturally there exists a torus decomposition $M_n=E(K)\cup \overline{N}$ of $M_n$. 
%
%\begin{remark}
%This manifold $M_n$ is diffeomorphic to a Brieskorn homology 3-sphere $\Sigma(2p,q,N)$ 
%where $N=|2pq n+1|$. 
%\end{remark}
% 
Here the fundamental group of $S^3\setminus T(p,q)$ has the presentation as follows; 
\[
\pi_1 (S^3\setminus T(p,q))=\langle x,y\ |\ x^{p} =y^ q\rangle.
\]
Furthermore $\pi_1(M_n)$ admits the presentation as follows;
\[
\pi_1(M_n)=\langle x,y\ |\ x^{p} =y^q, m l^n= 1\rangle
\]
where $m=x^{-r}y^s\ (r,s\in\Z,\ p s- q r=1)$ is a meridian of $T(p,q)$ 
and similarly $l=x^{-p}m^{p q}=y^{- q}m^{p q}$ is a longitude. 

%Here we take an irreducible representation $\rho:\pi_1(M_n)\rightarrow\SL(2;\C)$. 
%It is easy to see a given representation $\rho$ can be extended to $\pi_1(M_n)\rightarrow\SL(2;\C)$ as a representation 
%if and only if $\rho(ml^n)=E$. 
%Here $E$ is the identity matrix in $\SL(2;\C)$. 
%In this case by applying Proposition 2.5, 
%\[
%\tau_\rho(M_n)={\tau_\rho(E(K))}{\tau_\rho(\overline{N})}
%\]
%for any acyclic representation $\rho:\pi_1(M_n)\rightarrow\SL(2;\C)$. 
%
%Now we consider only irreducible representations of $\pi_1(M_n)$, 
%which is extended from the one on $\pi_1(E(K))$. 
It is seen \cite{Johnson, Kitano15-2} 
that the set of the conjugacy classes of the irreducible
representations of $\pi_1(M_n)$ in $\SL(2 ;\C)$ is finite. 
Any conjugacy class can be represented by $\rho_{(a,b,k)}:\pi_1(M_n)\rightarrow\SL(2;\C)$ 
for some triple $(a,b,k)$ 
such that
\begin{enumerate}
\item
$0<a<p,0<b< q, a\equiv b \ \text{mod } 2$,
\item
$0<k<N=|p q n+1|, k\equiv na\  \text{mod } 2$,
\item
$\mathrm{tr}(\rho_{(a,b,k)}(x))=2\cos \frac{a\pi }{p}$,
\item
$\mathrm{tr} (\rho_{(a,b,k)}(y))=2 \cos\frac{b\pi}{ q}$,
\item
$\mathrm{tr} (\rho_{(a ,b,k)}(m)) =2 \cos\frac{k\pi}{N}$.
\end{enumerate}

Furthermore Johnson computed $\tau_{\rho_{(a,b,k)}}(M_{n})$ as follows.

\begin{theorem}[Johnson]
\noindent
\begin{enumerate}
\item
A representation $\rho_{(a,b,k)}$ is acylic if and only if 
$a\equiv b\equiv 1$.
\item
For any acyclic representation ${\rho_{(a,b,k)}}$ with $a\equiv b{\equiv} 1$, 
then one has
\[
\tau_{\rho_{(a,b,k)}}(M_{n})
=\frac{1}{2\left(1-\cos\frac{a\pi}{p}\right) 
\left(1-\cos\frac{b\pi}{ q}\right)
\left(1+\cos\frac{p q k\pi }N\right)}.
\]
\end{enumerate}
\end{theorem}

%Recall the $n$-th Tchebychev polynomial $T_n(x)$ of the first kind can be defined 
%by expressing $\cos n\theta$ as a polynomial in $\cos\theta$. 
%We give a summary of these polynomials. 
%
%\begin{proposition}
%The Tchebychev polynomials have following properties.
%\begin{enumerate}
%\item
%$T_0(x)=1,T_1(x)=x$.
%\item
%$T_{-n}(x)=T_n(x)$.
%\item
%$T_n(-x)=\begin{cases}
%& -T_{n}(x)\ \text{ if $n$ is odd,}\\
%& T_{n}(x)\ \text{ if $n$ is even.}
%\end{cases}$
%\item
%$T_n(1)={1},T_n(-1)=(-1)^n$.
%\item
%$T_n(0)=\begin{cases}
%& 0\ \ \text{ if $n$ is odd,}\\
%& (-1)^{\frac n2}\text{ if $n$ is even.}
%\end{cases}$
%\item
%$T_{n+1}(x)=2xT_n-T_{n-1}(x)$.
%\item
%The degree of $T_n(x)$ is $n$.
%\item
%{$2T_m(x)T_n(x)=T_{m+n}(x)+T_{m-n}(x)$.}
%\end{enumerate}
%\end{proposition}
%
%He we put a short list of  $T_n(x)$.
%\begin{itemize}
%\item
%$T_0(x)=1$,
%\item
%$T_1(x)=x$,
%\item
%$T_2(x)=2x^2-1$,
%\item
%$T_3(x)=4x^3-3x$,
%\item
%$T_4(x)=8x^4-8x^2+1$,
%\item
%$T_5(x)=16x^5-20x^3+5x$,
%\item
%$T_6(x)=32x^6-48x^4+18x^2-1$.
%\end{itemize}

%%%%%%%%%%%%%%
\section{Main theorem}

%From this section, we consider the case for a $(p, q)$-torus knot. 
%Here $p,q$ are coprime integers. 
In this section 
we give a formula of the torsion polynomial 
${\sigma}_{(p, q,n)}(t)$ 
for $M_n=\Sigma(p, q,N)$ 
obtained by a $\frac1n$-Dehn surgery along $T(p,q)$. 
Now we define torsion polynomials as follows. 

\begin{definition}
A one variable polynomial ${\sigma}_{(p, q,n)}(t)$ is called the torsion polynomial 
of $M_n$ if the zero set coincides with the set 
of all non trivial values 
$\left\{\frac{1}{\tau_\rho(M_n)}\ |\ \tau_\rho(M_n)\neq 0\right\}$ 
and it satisfies the following normalization condition 
as \[
\sigma_{(p,q,n)}(0)=
\begin{cases}
&(-1)^{\frac{(N-1)p(q-1)}{8}}\ p\text{ is even}, q\text{ is odd},\\
&(-1)^{\frac{(N-1)(p-1)q}{8}}\ q\text{ is even}, q\text{ is odd},\\
&(-1)^{\frac{(N-1)(p-1)(q-1)}{8}}\ p,q\text{ are odd}, n\ \text{ is even},\\
&(-1)^{\frac{N(p-1)(q-1)}{8}}\ p,q\text{ are odd}, n\text{ is odd} 
\end{cases}
\]
where $N=|pqn+1|$. 
\end{definition}

\begin{remark}
\noindent
\begin{enumerate}
\item
For $M_{0}=S^{3}$, 
the torsion polynomial $\sigma_{(p,q,0)}(t)$ is defined 
by $\sigma_{(p,q,0)}(t)=1$. 
\item
In the case that $p=2p'$ is even and $p'$ is odd, 
then this normalization condition coincides with the one in \cite{Kitano15-2}.
\end{enumerate}
\end{remark}

From here assume $n\neq 0$. 
Recall Johnson's formula 
\[
\frac{1}{\tau_{\rho_{(a,b,k)}}(M_{n})}
={2\left(1-\cos\frac{a\pi}{p}\right) 
\left(1-\cos\frac{b\pi}{ q}\right)
\left(1+\cos\frac{p q k\pi}N\right)} 
\]
where $0<a<p,0<b< q, a\equiv b\equiv 1\text{ mod 2}, k\equiv n\text{ mod }2$. 
Here by putting  
\[
C_{(p, q,a,b)}=\left(1-\cos\frac{a\pi }{p}\right)\left(1-\cos\frac{b\pi }{ q}\right),
\]
one has 
\[
\frac{1}{\tau_{\rho_{(a,b,k)}}(M_n)}=4C_{(p,q,a,b)}\cdot \frac12\left(1+\cos\frac{p q k\pi}N\right).
\]

Main result is the following. 

\begin{theorem}
The torsion polynomial of $M_{n}$ is given by 
\[
{\sigma}_{(p, q,n)}(t)=
\prod_{(a,b)}Y_{(n,a,b)}(t)
\]
where
\[
Y_{(n, a,b)}(t)
=
\begin{cases}
&
\frac{
T_{{N+1}}(s)
-T_{{N-1}}(s)}{2(s^2-1)^2}\hskip 0.3cm (p\text{ or }q\text{ is even}, n>0),\\
&
-\frac{
T_{{N+1}}(s)
-T_{{N-1}}(s)}{2(s^2-1)^2}\ (p\text{ or }q\text{ is even}, n<0),\\
&
\frac{
T_{{N+1}}(s)
-T_{{N-1}}(s)}{2(s^2-1)^2}\hskip 0.3cm (p,q\text{ are odd}, n\text{ is even},n>0),\\
&
-\frac{
T_{{N+1}}(s)
-T_{{N-1}}(s)}{2(s^2-1)^2}\ (p,q\text{ are odd}, n\text{ is even},n<0).\\
&
T_N(s)\hskip 0.3cm (p,q,n\text{ are odd}).\\
%&
%-T_N(s)\hskip 0.3cm (p,q,n\text{ are odd}, n<0).
%&1\hskip 6.3cm (n=0).
\end{cases}
\]
Here 
\begin{itemize}
\item
$T_l(x)$ is the $l$-th Tchebychev polynomial of the first kind.
\item
$s=\displaystyle\frac{\sqrt{t}}{2\sqrt{C_{(p, q,a,b)}}}$.
\item
$C_{(p, q,a,b)}
=\left(1-\cos\frac{a\pi }{p}\right)\left(1-\cos\frac{b\pi }{ q}\right)$. 
\item
a pair of integers $(a,b)$ is satisfying the following conditions;
\begin{itemize}
\item
$0<a<p,0<b< q$, 
\item
$a\equiv b\equiv 1\text{ mod }2$. 
%\item
%$0<k<N, k\equiv n\text { mod }2$.
\end{itemize}
\end{itemize}
\end{theorem}

\begin{remark}
Recall that the $l$-th Tchebychev polynomial $T_l(x)$ is defined by 
$T_l(\cos\theta)=\cos(l\theta)$. 
\end{remark}

\begin{proof}
We consider the following;
\[
\begin{split}
X_n(x)
&=
\begin{cases}
&
\ \ \frac{T_{N+1}(x)-T_{N-1}(x)}{2(x^2-1)}\ (n>0)\\
&
-\frac{T_{N+1}(x)-T_{N-1}(x)}{2(x^2-1)}\ (n<0).\\
\end{cases}\\
X'_n(x)
&=T_{N}(x).
\end{split}
\]

%& 1. \hskip 5cm (n=0)
%we put 
%\[
%\begin{split}
%Y_{(n,a,b)}(t)
%&=X_n\left(\frac{\sqrt{t}}{2\sqrt{C_{(2p, q,a,b)}}}\right)\\
%&=\frac{T_{N+1}\left(\frac{\sqrt{t}}{2\sqrt{C_{(2p, q,a,b)}}}\right)
%-T_{N-1}\left(\frac{\sqrt{t}}{2\sqrt{C_{(2p, q,a,b)}}}\right)}
%{2\left(\left(\frac{\sqrt{t}}{2\sqrt{C_{(2p, q,a,b)}}}\right)^2-1\right)}\\
%&=\frac
%{T_{N+1}\left(\frac{\sqrt{t}}{2\sqrt{C_{(2p, q,a,b)}}}\right)
%-T_{N-1}\left(\frac{\sqrt{t}}{2\sqrt{C_{(2p, q,a,b)}}}\right)}
%{
%2\left(\frac{t}{4C_{(2p, q,a,b)}}-1\right)
%}\\
%&=2C_{(2p, q,a,b)}
%\frac
%{T_{N+1}\left(\frac{\sqrt{t}}{2\sqrt{C_{(2p, q,a,b)}}}\right)
%-T_{N-1}\left(\frac{\sqrt{t}}{2\sqrt{C_{(2p, q,a,b)}}}\right)}
%{
%t-4C_{(2p, q,a,b)}
%}\\
%\end{split}
%\]

First we assume $p=2p'$ is even. 
For the case that $p'$ is odd, then it is proved in \cite{Kitano15-2}. 
Then we suppose that $p'$ is even. Here $N=|2p'qn+1|$ is always odd. 

\noindent
\flushleft{Case 1:$p=2p'$, $p'$ is even and $n> 0$}

We modify one factor $(1+\cos\frac{2p' q k\pi}N)$ 
of $\displaystyle\frac{1}{\tau_\rho(M_n)}$ as follows. 
See \cite{Kitano15-2} for the proof. 

\begin{lemma}
The set $\{\cos\frac{2p' q  k\pi}{N}\ |\ 0<k<N, k\equiv n\text{ mod }2\}$ is equal to  
the set $\{\cos\frac{2p' k\pi }{N}\ |\ 0<k<\frac{N-1}{2}\}$. 
\end{lemma}

Now we can modify  
\[
\begin{split}
\frac12\left(1+\cos\frac{2p' k\pi}{N}\right)
&=\frac12\cdot2\cos^2\frac{2p' k\pi}{2N}\\
&=\cos^2\frac{p' k\pi}{N}.
\end{split}
\]

We put 
\[
z_k=\cos\frac{p' k\pi}{N}\ (1\leq k\leq N-1). 
\]
By the definition, it is seen  
\[
\begin{split}
z_{N-k}
&=\cos\frac{p'(N-k)\pi}{N}\\
&=\cos(p'\pi-\frac{ p'k\pi}{N})\\
&=z_k
\end{split}
\]
because $p'$ is even.

Therefore 
it is enough to consider only $z_k\ (1\leq k\leq \frac{N-1}{2})$. 

Now we substitute $x=z_k$ to $T_{N+1}(x)$. 
Then one has  
\[
\begin{split}
T_{N+1}(z_k)
&=\cos\left((N+1)\frac{p'k\pi}{N}\right)\\
&=\cos\frac{p'k\pi}{N}\\
&=z_k
\end{split}\]
and 
\[
\begin{split}
T_{N-1}(z_k)
&=\cos\left((N-1)\frac{p'k\pi}{N}\right)\\
&=\cos\frac{p'k\pi}{N}\\
&=z_k.
\end{split}
\]

Hence it holds 
\[
T_{N+1}(z_k)-T_{N-1}(z_k)=0.
\]
By properties of Tchebychev polynomials, 
it is seen that  
\begin{itemize}
\item
$T_{N+1}(1)-T_{N-1}(1)=0$,
\item
$T_{N+1}(-1)-T_{N-1}(-1)=0$.
\end{itemize}

We remark that 
the degree of $X_n(x)=\frac{T_{N+1}(x)-T_{N-1}(x)}{2(x^2-1)}$ is $N-1$ 
and $z_{1},\cdots,z_{\frac{N-1}{2}}$ are zeros. 
Because both of $T_{N+1}(x)$ and $T_{N-1}(x)$ are even functions, 
then $-z_{1},\cdots,-z_{\frac{N-1}{2}}$ are also zeros of $X_{n}(x)$. 
Hence $X_n(x)$ is a functions of $x^2$. 
Here by replacing 
%$x^2$ by $\frac{t}{4C_{(p, q,a,b)}}$, namely 
$x$ by $\frac{\sqrt{t}}{2\sqrt{C_{(p, q,a,b)}}}$, 
the degree of $Y_{(n,a,b)}(t)$ is $\frac{N-1}{2}$, 
and the roots of $Y_{(n,a,b)}(t)$ 
are 
\[
%\left\{
4C_{(p, q,a,b)}z_k^2
=4C_{(p, q,a,b)}\cos^2\frac{\pi k}{N}
\ \ \ 
\left(0<k<\frac{N-1}{2}\right),
\]
which are all non trivial values of 
$\frac{1}{\tau_{\rho_{(a,b,k)}}(M_n)}$.

Here we check the normalization condition. 
By the definition of $Y_{(n,a,b)}(t)$ and properties of $T_{N+1}(x),T_{N-1}(x)$, 
one has
\[
\begin{split}
Y_{(n,a,b)}(0)
&=\frac{T_{N+1}(0)-T_{N-1}(0)}{2(0-1)}\\
&=-\frac{(-1)^{\frac{N+1}{2}}-(-1)^{\frac{N-1}{2}}}{2}\\
&=(-1)^{\frac{N-1}{2}}.
\end{split}
\]
Hence it can be seen 
\[
\begin{split}
\sigma_{(p,q,n)}(0)
&=\prod_{(a,b)}(-1)^\frac{N-1}{2}\\
&=\prod_{(a,b)}\left((-1)^\frac{N-1}{2}\right)^{\frac{p(q-1)}{4}}\\
&=(-1)^{\frac{(N-1)p(q-1)}{8}}. 
\end{split}
\]
Therefore we obtain the formula. 

\noindent
\flushleft{Case 2:$p=2p'$ and $n<0$}

In this case we modify $N=|2p'qn+1|=2p'q|n|-1$. 
By the same arguments, it is easy to see the claim of the theorem is proved. 
%Therefore the proof completes. 

Next assume both of $p,q$ are odd integers. 

\flushleft{Case 3:$p,q$ are odd and $n$ is even}

If $n$ is even, then $N=|pqn+1|$ is odd. 
Then the similar arguments in \cite{Kitano15-2} work well. 
Then it can be proved. 

\flushleft{Case 4:$p,q$ are odd and $n$ is odd}

Suppose $n$ is positive. 
First note that $N=|pqn+1|$ is even. 
%%%%%%%%%%%%%%%%%%%%%
We can modify one factor $(1+\cos\frac{p q k\pi}N)$ 
of $\displaystyle\frac{1}{\tau_\rho(M_n)}$ as follows. 
It is clear because $(q,N)=1$. 

\begin{lemma}
The set $\{\cos\frac{p q  k\pi}{N}\ |\ 0<k<N, k\equiv n\text{ mod }2\}$ is equal to  
the set $\{\cos\frac{p k\pi }{N}\ |\ 0<k<{N}, k\equiv 1\text{ mod }2\}$. 
\end{lemma}

Now we can modify  
\[
\begin{split}
\frac12\left(1+\cos\frac{p k\pi}{N}\right)
&=\frac12\cdot2\cos^2\frac{p k\pi}{2N}\\
&=\cos^2\frac{p k\pi}{2N}.
\end{split}
\]

We put 
\[
z'_k=\cos\frac{p k\pi}{2N}\ (1\leq k\leq N-1,\ k\equiv 1\ \text{mod}\ 2). 
\]

Here we subsitute 
$x=z'_k\ (1\leq k\leq \frac{N-1}{2},\ k\equiv 1\ \text{mod}\ 2)$ 
to $T_{N}(x)$. 
Then one has 
\[
\begin{split}
T_{N}(z'_k)
&=\cos\left(\frac{N(p k\pi)}{2N}\right)\\
&=\cos\left(\frac{p k\pi}{2}\right)\\
&=0 
\end{split}
\]
because $pk$ is odd. 

Similarly it can be also seen that 
\[
T_{N}(-z'_k)=0.
\]

We mention that 
the degree of $\displaystyle X'_n(x)=T_{N}(x)$ is $N$ 
and $\pm z'_{1},\cdots,\pm z'_{N-1}$ are the zeros. 
Because $X'_n(x)$ is a functions of $x^2$. 
Here by replacing %$x^2$ by $\frac{t}{4C_{(p, q,a,b)}}$, namely 
$x$ by $\frac{\sqrt{t}}{2\sqrt{C_{(p, q,a,b)}}}$, 
Here it holds that its degree of $Y_{(n,a,b)}(t)$ is $\frac{N-1}{2}$, 
and the roots of $Y_{(n,a,b)}(t)$ 
are 
\[
4C_{(p, q,a,b)}{z'_k}^2
=4C_{(p, q,a,b)}\cos^2\frac{\pi k}{N}
\ \ \left(0<k<\frac{N-1}{2}\right),
\] 
which are all non trivial values of $\frac{1}{\tau_{\rho_{(a,b,k)}}(M_n)}$.

Finally we can check the normalization condition as follows. 
By the definition of $Y_{(n,a,b)}(t)$, one has 
\[
\begin{split}
Y_{(n,a,b)}(0)
&=T_{N}(0)\\
&=(-1)^{\frac{N}{2}}
\end{split}
\]
and 
\[
\begin{split}
\sigma_{(p,q,n)}(0)
&=\prod_{(a,b)}(-1)^\frac{N}{2}\\
&=\left(
(-1)^\frac{N}{2}
\right)^{\frac{(p-1)(q-1)}{4}}\\
&=(-1)^{\frac{N(p-1)(q-1)}{8}}. 
\end{split}
\]

Therefore we obtain the formula. 

In the case that $n$ is negative, then it can be proved by similar arguments. 
Therefore this completes the proof.
\end{proof}

\begin{remark}
By defining as $X_{0}(t)=1$, it implies $Y_{(0,a,b)}(t)=1$. 
Then the above statement is true for $n=0$.
\end{remark}

By direct computation, one obtains the following corollary. 

\begin{corollary}
The degree $\mathit{deg}(\sigma_{(p,q,n)}(t))$ is given by 
\[
\mathit{deg}(\sigma_{(p,q,n)}(t))
=
\begin{cases}
&\frac{(N-1)p(q-1)}{8}\ \ (p\text{ even},q\text{ odd}),\\
&\frac{(N-1)(p-1)q}{8}\ \ (p\text{ odd},q\text{ even}),\\
&\frac{(N-1)(p-1)(q-1)}{8}\ \ (p,q\text{ odd},n\text{ even}),\\
&\frac{N(p-1)(q-1)}{8}\ \ (p,q\text{ odd},n\text{ odd}).\\
%&\frac{N(p-1)(q-1)}{8}\\
\end{cases}
\]
\end{corollary}

We mention the 3-term relations. 
For each factor of $Y_{(n,a,b)}(t)$ 
of $\sigma_{(p,q,n)}(t)$, there exists the following relation. 
%The following proposition holds by the same argument in \cite{Kitano15-2}. 

\begin{proposition}
\noindent
\begin{enumerate}
\item
Assume one of $p$ and $q$ is even.  
For any $n$, it holds that 
\[
Y_{(n+1,a,b)}(t)=D(t)Y_{(n,a,b)}(t)-Y_{(n-1,a,b)}(t)
\]
where 
$D(t)=2T_{p q }\left(\frac{\sqrt{t}}{2\sqrt{C_{p, q,a,b}}}\right)$. 
\item
Assume both of $p,q$ are odd. 
For any $n$, it holds that 
\[
Y_{(n+2,a,b)}(t)=D(t)Y_{(n,a,b)}(t)-Y_{(n-2,a,b)}(t)
\]
where 
$D(t)=2T_{2p q }\left(\frac{\sqrt{t}}{2\sqrt{C_{2p, q,a,b}}}\right)$. 
\end{enumerate}
\end{proposition}

\begin{proof}
Here we need to consider $N=|pqn+1|$ is a function of $n\in\Z$ for fixed $p,q$. 
Then we write $N(n)$ for $N$ in this proof. 

The proof for the first case is essentially the same one for the 3-term relations \cite{Kitano15-2}. 
We give the proof only for the second case. 

Recall the following property of Tchebychev polynomials
\[
2T_m(x)T_n(x)=T_{m+n}(x)+T_{m-n}(x)
\]
for any $m,n\in\Z$. 

\noindent
\flushleft{Case 1: $n$ is even}

If $n>0$ one has
\[
\begin{split}
2T_{2p q}(x)X_n(x)
&=2T_{2p q }(x)\frac{T_{N(n)+1}(x)-T_{N(n)-1}(x)}{2(x^2-1)}\\
&=\frac{T_{(pqn+1)+1+2pq}(x)+T_{(pqn+1)+1-2pq}(x)
-(T_{(pqn+1)-1+2pq}(x)+T_{(pqn+1)-1-2pq}(x))}{2(x^2-1)}\\
&=\frac{T_{pq(n+2)+1+1}(x)-T_{pq(n+2)+1-1}(x)
+T_{pq(n-2)+1+1}(x)-T_{pq(n-2)+1-1}(x)}{2(x^2-1)}\\
&=\frac{T_{N(n+2)+1}(x)-T_{N(n+2)-1}(x)
+T_{N(n-2)+1}(x)-T_{N(n-2)-1}(x)}{2(x^2-1)}\\
&=X_{n+2}(x)+X_{n-2}(x).
\end{split}
\]
Therefore it can be seen that 
\[
X_{n+2}(x)=2T_{2p q }(x)X_n(x)-X_{n-2}(x)
\]
and 
\[
Y_{(n+2,a,b)}(t)
=2T_{2p q }\left(\frac{\sqrt{t}}{2\sqrt{C_{(2p,q,a,b)}}}\right)
Y_{(n,a,b)}(t)-Y_{(n-2,a,b)}(t).
\]

If $n<0$, it can be also proved by the above argument. 

\noindent
\flushleft{Case 2: $n$ is odd}

If $n>0$, one has 
\[
\begin{split}
2T_{2p q}(x)X'_n(x)
&=2T_{2p q }(x)T_{N(n)}(x)\\
&=T_{pqn+1+2pq}(x)+T_{pqn+1-2pq}(x)\\
&=T_{pq(n+2)+1}(x)+T_{pq(n-2)+1}(x)\\
&=T_{N(n+2)}(x)+T_{N(n-2)}(x)\\
&=X'_{n+2}(x)+X'_{n-2}(x).
\end{split}
\]
Therefore it can be seen that 
\[
X'_{n+2}(x)=2T_{2p q }(x)X'_n(x)-X'_{n-2}(x)
\]
and 
\[
Y_{(n+2,a,b)}(t)
=2T_{2p q }\left(\frac{\sqrt{t}}{2\sqrt{C_{(2p,q,a,b)}}}\right)
Y_{(n,a,b)}(t)-Y_{(n-2,a,b)}(t).
\]

If $n<0$, it can be also proved. 

This completes the proof of this proposition.
\end{proof}

%In the case that $p,q$ are odd, we have the following. 
%
%\begin{proposition}
%If $n$ is odd, then 
%\[
%D(s)Y_{(n,a,b)}(s)-Y_{(n-2,a,b)}(s))=2(s^2-1)(Y_{(n+1,a,b)}(s)-Y_{(n-1,a,b)}(s))
%\]
%where 
%$D(s)=2T_{p q }\left(\frac{\sqrt{s}}{2\sqrt{C_{p, q,a,b}}}\right)$.
%\end{proposition}
%%%%%%%%%

\section{examples}

Finally we give some examples. 

\begin{example}
Put $p=4, q=3$. Now $N=|12n+1|$. 
In this case $(a,b)=(1,1),(3,1)$. 
%Then we see 
%\[
%\begin{split}
%C_{4,3,1,1}
%&=\left(1-\cos\frac\pi 4\right)\left(1-\cos\frac{\pi}{3}\right)\\
%&=\frac14(2-\sqrt{2}),\\
%C_{4,3,3,1}
%&=\left(1-\cos\frac{3\pi}{4}\right)\left(1-\cos\frac{\pi}{3}\right)\\
%&=\frac14(2+\sqrt{2}).
%\end{split}
%\]
By applying the main theorem, one has
\[
\begin{split}
\sigma_{(4,3,-1)}(t)
&=34359738368t^{10}-77309411328t^{9}+66840428544t^{8}\\
&-28655484928t^{7}+6677331968 t^{6}-882900992 t^{5}+66371584t^{4}\\
&-2723840 t^{3}+55680t^{2}-480t+1.\\
\sigma_{(4,3,0)}(t)&=1.\\
\sigma_{(4,3,1)}(t)&=4398046511104t^{12}-12094627905536t^{11}+13434657701888t^{10}\\
&-7859790151680t^{9}+2670664351744t^{8}-552909930496t^{7}\\
&+71319945216t^{6}-5727322112 t^{5}+278757376t^{4}\\
&-7741440 t^{3}+110208t^{2}-672t+1.
\end{split}
\]
\end{example}

\begin{example}
Put $p=3, q=5$. Now $N=|15n+1|$. 
In this case $(a,b)=(1,1),(1,3)$. 
For any odd number $n$, 
one has 
\[
\begin{split}
\sigma_{(3,5,n)}(t)
&=Y_{(n,1,1)}(t) Y_{(n,1,3)}(t)\\
&=T_{N}
\left(\frac{\sqrt{t}}{2\sqrt{C_{(3,5,1,1)}}}\right) Y_{N}\left(\frac{\sqrt{t}}{2\sqrt{C_{(3,5,1,3)}}}\right). 
\end{split}
\]
%Then we see 
%\[
%\begin{split}
%C_{3,5,1,1}
%&=\left(1-\cos\frac\pi 3\right)\left(1-\cos\frac{\pi}{5}\right),\\
%C_{3,5,1,3}
%&=\left(1-\cos\frac{\pi}{3}\right)\left(1-\cos\frac{3\pi}{5}\right).\\
%\end{split}
%\]
By applying the main theorem, we obtain
\[
\begin{split}
\sigma_{(3,5,-1)}(t)
&=
18014398509481984 t^{14}-47287796087390208 t^{13}+51721026970583040 t^{12}\\
&-30847898228883456 t^{11}+11085001353330688 t^{10}-2520389888507904
t^9\\
&+372923420377088 t^8-36436086620160 t^7+2352597696512 t^6\\
&-98837200896 t^5+2605023232 t^4-40341504 t^3+329280 t^2-1176 t+11.\\
\sigma_{(3,5,0)}(t)
&=1.\\
\sigma_{(3,5,1)}(t)
&=
4611686018427387904 t^{16}-13835058055282163712
t^{15}\\
&+17726168133330272256 t^{14}-12754194144713244672 t^{13}\\
&+5718164151876976640 t^{12}-1682516673287946240 t^{11}\\
&+334779300425236480 t^{10}-45872724622442496
t^9\\
&+4367893693202432 t^8+-288911712583680 t^7\\
&+13126896451584 t^6-399582953472 t^5\\
&+7798652928 t^4-90832896 t^3+563200 t^2-1536 t+1.
\end{split}
\]

\end{example}

%%%%%%%%%%%%%%%%%%%%%%%%%%%%%%%%%%%%%%%%%%%%%%%%%%%%%%%%%%%
%
%Finally we mention some problems. 
%\begin{problem}
%\noindent
%\begin{itemize}
%\item
%How strong the set of Reidemeister torsions and the torsion polynomial are in %general ?
%\item
%Can this torsion polynomial be computed for any $(p,q)$-torus knot ? 
%\item
%Can this torsion polynomial be computed for any homology 3-sphere with the finite %set of $\{\tau_{\rho}\}$ ? 
%\item
%How it can be treated for a 3-manifold with the infinite set of $\{\tau_{\rho}\}$ ?
%\end{itemize}
%\end{problem}

%%%%%%%%%%%%%%%%%%%%%%%%%%%%%%%%%%%%%%%%%%%%%%%%%%%%%%%%%%

%%%%%%%%%%%%%%%%%%%%%%%%%%%%%%%%

\end{document}